\documentclass[a4wide,12pt]{amsart}
\usepackage{scrextend} 
\usepackage{setspace}
\usepackage[top=.8in, bottom=.7in]{geometry}
\usepackage{graphicx}
\graphicspath{ {./figures/} }
\usepackage{subcaption}
\usepackage{amsmath, amsthm, amssymb, mathrsfs}
\usepackage{dsfont}
\usepackage[dvipsnames]{xcolor}
\usepackage{natbib}
\usepackage{blindtext}
\usepackage[english]{babel}
\usepackage{latexsym}
\usepackage{enumerate}
\usepackage{mathtools}

\allowdisplaybreaks

\newtheorem{definition}{Definition}[section]
\newtheorem{theorem}{Theorem}[section]

\newtheorem{lemma}{Lemma}[section]
\newtheorem{remark}{Remark}[section]

\title[Numerical Approximation of Stochastic Linear Schr\"{o}dinger equation]{Finite Element Approximations of Stochastic Linear Schr\"{o}dinger equation driven by additive Wiener noise}

\author[S. Bhar]{Suprio Bhar}
\address{Indian Institute of Technology Kanpur, Kanpur-208016, India,}
\email{suprio@iitk.ac.in}

\author[M. Biswas]{Mrinmay Biswas}
\address{Indian Institute of Technology Kanpur, Kanpur-208016, India,}
\email{ mbiswas@iitk.ac.in}

\author[M. Prasad]{Mangala Prasad}
\address{Indian Institute of Technology Kanpur, Kanpur-208016, India,}
\email{ mangalap21@iitk.ac.in}

\allowdisplaybreaks
\begin{document}
	
	%
	
	
	
	\date{\today}
	
	\thanks{}
	
	\begin{abstract}
		In this article, we have analyzed semi-discrete finite element approximations of the Stochastic linear Schr\"{o}dinger equation in a bounded convex polygonal domain driven by additive Wiener noise. We use the finite element method for spatial discretization and derive an error estimate with respect to the discretization parameter of the finite element approximation. Numerical experiments have also been performed to support theoretical bounds. 
	\end{abstract}
	
	\maketitle
	\section{Introduction}\label{s1}
In this work, we study the finite element approximation of the stochastic linear 
Schr\"{o}dinger equation driven by additive noise. The model problem is 
\begin{equation}\label{eqn1}
\begin{aligned}
&du +i\,\Delta u\,dt = dW_1 + i\, dW_2 , \qquad (t,x)\in (0,\infty)\times \mathcal{O},\\[2mm]
&u(t,x)=0 , \qquad (t,x)\in (0,\infty)\times \partial\mathcal{O},\\[2mm]
&u(0,x)=u_0(x), \qquad x\in \mathcal{O},
\end{aligned}
\end{equation}
where $\mathcal{O}\subset \mathbb{R}^d$, $d=1,2,3$, is a bounded convex polygonal. The processes $\{W_j(t)\}_{t\ge 0}$, $j=1,2$, are two independent 
$L^2(\mathcal{O})$-valued Wiener processes defined on a filtered probability space 
$(\Omega,\mathcal{F},P,\{\mathcal{F}_t\}_{t\ge 0})$. Throughout the paper, 
$(\cdot,\cdot)$ and $\|\cdot\|$ denote the inner product and norm in $L^2(\mathcal{O})$, 
respectively, and the initial datum $u_0$ is assumed to be $\mathcal{F}_0$-measurable.

The Schr\"{o}dinger equation plays a central role in quantum mechanics, governing the 
time evolution of the wave function of a physical system \cite{griffiths}. Inspired by 
de~Broglie's hypothesis that every particle exhibits wave-like behaviour, Schr\"{o}dinger 
derived the celebrated equation that now carries his name. The model accurately predicts 
atomic bound states and agrees well with experimental observations; see 
\cite{Whittaker1989} for further details. The deterministic Schr\"{o}dinger equation and 
its semilinear variants have been extensively investigated in the classical literature; 
see, for example, \cite{courantL,dautray}.

From an applied and computational perspective, numerical approximation of stochastic 
Schr\"{o}dinger equations is of significant interest. Explicit analytical solutions are 
rarely available, and in higher spatial dimensions, finite element Galerkin methods offer 
a flexible and robust framework for approximating solutions of partial differential 
equations. This makes finite element techniques a natural choice for analysing stochastic 
extensions of Schr\"{o}dinger-type models.

The numerical analysis of stochastic partial differential equations (SPDEs) has seen 
substantial progress over recent decades. A large body of work addresses the stochastic 
heat equation; see, for example, 
\cite{pchow, kovacs9bit, Igyongy, kovacs10NA, JBWalsh, YYanBit4, YYanSiam5}. 
Similarly, the stochastic wave equation has been widely studied; we refer to 
\cite{DCohen16, DCohen13, Dcohen12, Kovacs10Siam, Amartin, MR2224753, CRoth, JBWalsh6} 
and the references therein. 

In contrast, numerical studies of stochastic Schr\"{o}dinger equations remain rather 
limited. Existing results mainly focus on finite difference schemes or spectral Galerkin 
approximations; see \cite{chuchuhong,jianbo, Feng}. These works provide strong 
convergence results under various assumptions, but the finite element method despite its 
flexibility and importance in higher dimensions has received considerably less attention 
in this context.

The main contribution of the present paper is the derivation of explicit strong 
convergence rates for a finite element Galerkin discretization of 
\eqref{eqn1}. Our analysis relies on Ritz and $L^2$ projections. The results fill a gap in the existing literature and 
provide a systematic foundation for the finite element approximation of stochastic 
Schr\"{o}dinger equations with additive noise.

\medskip

\noindent\textbf{Organization of the paper.}
Section \ref{main_results} introduces the model and summarizes the main results.  
Section~\ref{s2} collects preliminaries on Hilbert--Schmidt operators, 
infinite-dimensional Wiener processes, semigroup representations, and basic finite 
element tools.  
Section~\ref{prf_main} contains the proofs of the main results.  
Finally, Section~\ref{S_4} presents numerical experiments illustrating the strong 
convergence behaviour of the finite element scheme.
\section{Main Results}\label{main_results}

In this section we present our main  results for the spatial finite element 
approximation of the stochastic linear Schr\"{o}dinger equation \eqref{eqn1}. 
We begin with the deterministic case, where the noise terms $W_1$ and $W_2$ are absent, 
and derive error estimates for both the homogeneous and nonhomogeneous problems. 
These deterministic bounds form the foundation for the analysis of the fully 
stochastic equation, which is treated in the second part of this section.

\subsection{\bf Deterministic Version} 
We first study the deterministic non-homogeneous formulation and derive the corresponding finite element error bounds.
\subsubsection{Nonhomogeneous System}

We first study the spatially semidiscrete finite element approximation for the 
deterministic nonhomogeneous linear Schr\"odinger equation
\begin{align} \label{system2}
\begin{cases}
\dfrac{du}{dt} +i\,\Delta u = f, & \text{in } (0,\infty)\times\mathcal{O},\\[1.5mm]
u = 0, & \text{on } (0,\infty)\times \partial\mathcal{O},\\[1.5mm]
u(0,x) = u_0(x), & \text{in } \mathcal{O},
\end{cases}
\end{align}
where $\mathcal{O}\subset \mathbb{R}^d$, $d=1,2,3$, is a bounded convex 
polygonal domain with boundary $\partial\mathcal{O}$.  
Here $u, f : [0,\infty)\times\overline{\mathcal{O}}\to\mathbb{C}$ and 
$u_0:\overline{\mathcal{O}}\to\mathbb{C}$ are complex-valued functions.

To analyse the system componentwise, we write
\[
u(t) = u_1(t) + i\,u_2(t), \qquad 
f(t) = f_1(t) + i\,f_2(t), \qquad 
u_0 = u_{0,1} + i\,u_{0,2},
\]
where $u_1, u_2, f_1, f_2, u_{0,1}, u_{0,2}$ denote the real and imaginary parts 
of $u$, $f$, and $u_0$, respectively.  
With this notation, \eqref{system2} is equivalent to the real system
\begin{align} \label{system3}
\begin{cases}
\dot{u}_1 - \Delta u_2 = f_1, 
& \text{in } (0,\infty)\times\mathcal{O},\\[1mm]
\dot{u}_2 + \Delta u_1 = f_2, 
& \text{in } (0,\infty)\times\mathcal{O},\\[1mm]
u_1 = 0,\; u_2 = 0, 
& \text{on } (0,\infty)\times\partial\mathcal{O},\\[1mm]
u_1(0) = u_{0,1},\quad 
u_2(0) = u_{0,2}, 
& \text{in } \mathcal{O}.
\end{cases}
\end{align}

We assume that 
\[
f_1, f_2 \in L^2(0,\infty; \dot{H}^0),
\qquad 
u_{0,1}, u_{0,2} \in \dot{H}^0.
\]
Throughout, the dot ``$\cdot$’’ denotes the time derivative 
$\frac{\partial}{\partial t}$.  
The definition of the fractional spaces $\dot{H}^\alpha$, $\alpha\in\mathbb{R}$, 
is provided in Subsection~\ref{s2.1}.
\begin{definition}[Weak solution of system \eqref{system3}]
A pair $(u_1,u_2)^{\mathrm{T}} \in L^2(0,\infty;\dot{H}^1)\times L^2(0,\infty;\dot{H}^1)$ with 
$(\dot{u}_1,\dot{u}_2)^{\mathrm{T}} \in L^2(0,\infty;\dot{H}^0) \times L^2(0,\infty;\dot{H}^0)$ 
is called a weak solution of \eqref{system3} if for a.e.\ $t>0$,
\begin{equation}\label{system4}
\begin{split}
(\dot{u}_1(t),v_1) + (\nabla u_2(t),\nabla v_1) &= (f_1(t),v_1),\\
(\dot{u}_2(t),v_2) - (\nabla u_1(t),\nabla v_2) &= (f_2(t),v_2),
\qquad \forall v_1,v_2 \in \dot{H}^1,\\
u_1(0) = u_{0,1},\qquad 
u_2(0) = u_{0,2}.
\end{split}
\end{equation}
\end{definition}

The existence and uniqueness of a weak solution of \eqref{system3} is standard; 
see \cite{pgrisvard}. Our goal is to construct a spatially semidiscrete finite 
element approximation of the weak solution of \eqref{system3}.  

The semidiscrete analogue of \eqref{system4} is to find 
$u_{h,1}(t), u_{h,2}(t) \in V_h$ (the definition of $V_h$ is given in 
Subsection~\ref{s2.2}) such that
\begin{equation}\label{system5}
\begin{split}
(\dot{u}_{h,1}(t),\chi_1) + (\nabla u_{h,2}(t),\nabla \chi_1) 
&= (f_1(t),\chi_1),\\
(\dot{u}_{h,2}(t),\chi_2) - (\nabla u_{h,1}(t),\nabla \chi_2) 
&= (f_2(t),\chi_2),\qquad \forall \chi_1,\chi_2 \in V_h, \ t>0,\\
u_{h,1}(0) = u_{h,0,1},\qquad 
u_{h,2}(0) = u_{h,0,2},
\end{split}
\end{equation}
with initial values $u_{h,0,1}, u_{h,0,2} \in V_h$. 
Existence and uniqueness of solutions to \eqref{system5} follow from standard 
results; see \cite{ciarlet}.

\medskip

Setting $\chi_i = \Lambda_h^{\alpha} u_{h,i}(t)$, $i=1,2$, where 
$\alpha\in\mathbb{R}$ and $\Lambda_h$ is defined in Subsection~\ref{s2.2},  
we obtain the following stability estimate.

\begin{theorem}\label{est_appr_1}
Let $\alpha\in\mathbb{R}$ and let $u_{h,1},u_{h,2}$ be the solution of 
\eqref{system5} with initial data 
$(u_{h,1}(0),u_{h,2}(0))^{\mathrm{T}} 
= (u_{h,0,1},u_{h,0,2})^{\mathrm{T}}$.  
Then, for all $t\ge 0$,
\begin{equation}\label{eqn2}
\begin{split}
\|u_{h,1}(t)\|_{h,\alpha} + \|u_{h,2}(t)\|_{h,\alpha}
&\le C\Big( \|u_{h,0,1}\|_{h,\alpha} + \|u_{h,0,2}\|_{h,\alpha}  \\
&\qquad + \int_0^t 
\big( \|\mathcal{P}_h f_1(s)\|_{h,\alpha} 
    + \|\mathcal{P}_h f_2(s)\|_{h,\alpha} \big)\,ds \Big).
\end{split}
\end{equation}
\end{theorem}

\medskip

We now state the deterministic finite element error estimate for the nonhomogeneous 
problem. Its proof is given in Section~\ref{prf_main}.

\begin{theorem}\label{det_err1}
Let $(u_1,u_2)$ and $(u_{h,1},u_{h,2})$ be the weak solution of 
\eqref{system4} and the semidiscrete finite element solution of 
\eqref{system5}, respectively.  
Define the errors $e_i = u_{h,i} - u_i$, $i=1,2$.  
Then, for all $t\ge 0$,
\begin{equation}\label{main_est3}
\begin{split}
\|e_1(t)\|
&\le 
C\big( \|u_{h,0,1} - \mathcal{R}_h u_{0,1}\|
     + \|u_{h,0,2} - \mathcal{R}_h u_{0,2}\| \big) \\
&\qquad 
+ C h^2\left(
\int_0^t \|\dot{u}_1(s)\|_2\,ds
+ \int_0^t \|\dot{u}_2(s)\|_2\,ds
+ \|u_1(t)\|_2
\right),
\end{split}
\end{equation}
and
\begin{equation}\label{main_est4}
\begin{split}
\|e_2(t)\|
&\le 
C\big( \|u_{h,0,1} - \mathcal{R}_h u_{0,1}\|
     + \|u_{h,0,2} - \mathcal{R}_h u_{0,2}\| \big) \\
&\qquad 
+ C h^2\left(
\int_0^t \|\dot{u}_1(s)\|_2\,ds
+ \int_0^t \|\dot{u}_2(s)\|_2\,ds
+ \|u_2(t)\|_2
\right).
\end{split}
\end{equation}
\end{theorem}

In the above, $\mathcal{P}_h$ denotes the $L^2$-orthogonal projection from 
$\dot{H}^0$ onto $V_h$, $\mathcal{R}_h$ denotes the Ritz projection from 
$\dot{H}^1$ onto $V_h$, and the norms $\|\cdot\|_{h,\alpha}$ are defined in 
Subsection~\ref{s2.2}.
\subsubsection{Homogeneous System}

In this subsection, we consider the deterministic homogeneous linear 
Schr\"odinger equation and its spatial finite element approximation.  
We begin with the continuous problem
\begin{equation}\label{system6}
\dot{u}(t) + i \Delta u(t) = 0,\qquad t>0,
\qquad 
u(0)=u_0.
\end{equation}
It is well known that solutions of \eqref{system6} preserve energy in various 
Sobolev norms. In particular, differentiating the equation $r$ times in space 
yields the identity
\begin{equation}\label{eqn3}
\|D^r u_1(t)\|_{\alpha}^2 + \|D^r u_2(t)\|_{\alpha}^2
= 
\|\Lambda^r u_{0,1}\|_{\alpha}^2 + \|\Lambda^r u_{0,2}\|_{\alpha}^2,
\qquad t\ge 0,
\end{equation}
where $u_1=\mathrm{Re}(u)$ and $u_2=\mathrm{Im}(u)$, and 
$u_{0,1}=\mathrm{Re}(u_0)$, $u_{0,2}=\mathrm{Im}(u_0)$.

\medskip

Writing the homogeneous equation in real and imaginary components gives the system
\begin{equation}\label{sys_hom}
\frac{d}{dt}
\begin{bmatrix}
u_1 \\ u_2
\end{bmatrix}
=
\begin{bmatrix}
0 & -\Lambda \\
\Lambda & 0
\end{bmatrix}
\begin{bmatrix}
u_1 \\ u_2
\end{bmatrix},
\qquad t>0,
\qquad
\begin{bmatrix}
u_1(0) \\ u_2(0)
\end{bmatrix}
=
\begin{bmatrix}
u_{0,1} \\ u_{0,2}
\end{bmatrix},
\end{equation}
where $\Lambda = -\Delta$.

\medskip

System \eqref{sys_hom} may be written abstractly as
\begin{equation}\label{abs_hom}
\dot{X}(t) = A X(t),\qquad t>0,
\qquad
X(0)=X_0,
\end{equation}
where
\[
A=
\begin{bmatrix}
0 & -\Lambda \\
\Lambda & 0
\end{bmatrix},\qquad
X(t)=
\begin{bmatrix}
u_1(t) \\ u_2(t)
\end{bmatrix},\qquad
X_0=
\begin{bmatrix}
u_{0,1} \\ u_{0,2}
\end{bmatrix}.
\]

In this framework, the weak solution of \eqref{abs_hom} is given by
\begin{equation}
X(t) = E(t) X_0,\qquad t\ge 0,
\end{equation}
where the semigroup $\{E(t)\}_{t\ge 0}$ is generated by $(A,D(A)$ in $H^{\alpha}$  and satisfies
\[
E(t) = e^{tA}
=
\begin{bmatrix}
C(t) & -S(t)\\
S(t) & C(t)
\end{bmatrix},\qquad 
C(t)=\cos(t\Lambda),\quad S(t)=\sin(t\Lambda).
\]
Further details are provided in Subsection~\ref{s2.1}.

\medskip

The corresponding finite element approximation of \eqref{abs_hom} is defined by
\begin{equation}\label{homapp_eqn}
\dot{X}_h(t) = A_h X_h(t),\qquad t>0,
\qquad
X_h(0)=X_{h,0},
\end{equation}
where
\[
A_h=
\begin{bmatrix}
0 & -\Lambda_h\\
\Lambda_h & 0
\end{bmatrix},\qquad
X_h(t)=
\begin{bmatrix}
u_{h,1}(t)\\
u_{h,2}(t)
\end{bmatrix},\qquad
X_{h,0}=
\begin{bmatrix}
u_{h,0,1}\\
u_{h,0,2}
\end{bmatrix}.
\]

The finite-dimensional operator $A_h$ generates a $C_0$-semigroup 
$\{E_h(t)\}_{t\ge 0}$ on $V_h\times V_h$, given by
\[
E_h(t)=e^{tA_h}
=
\begin{bmatrix}
C_h(t) & -S_h(t)\\
S_h(t) & C_h(t)
\end{bmatrix},\qquad 
C_h(t)=\cos(t\Lambda_h),\qquad S_h(t)=\sin(t\Lambda_h).
\]
For example, using the orthonormal eigenpairs 
$\{(\lambda_{h,j},\phi_{h,j})\}_{j=1}^{N_h}$ of $\Lambda_h$ 
(with $N_h=\dim(V_h)$), we have for $v_h\in V_h$,
\[
C_h(t)v_h
=
\cos(t\Lambda_h)v_h
=
\sum_{j=1}^{N_h}
\cos(t\lambda_{h,j})(v_h,\phi_{h,j})\,\phi_{h,j},\qquad t\ge 0.
\]

The following estimate for the homogeneous problem will be used later in the 
analysis of the stochastic equation (see Theorem~\ref{stoch_err1}).
\begin{theorem}\label{det_err2}
Let $\beta\in[0,4]$ and $X_0=(u_{0,1},u_{0,2})^{\mathrm{T}}\in H^\beta$.  
For $t\ge 0$, define the linear operators $F_h(t),G_h(t):H^\beta \to \dot{H}^0$ by
\begin{align*}
F_h(t)X_0 
&:= \big(C_h(t)\mathcal{P}_h - C(t)\big)u_{0,1}
    - \big(S_h(t)\mathcal{P}_h - S(t)\big)u_{0,2},\\[1mm]
G_h(t)X_0
&:= \big(S_h(t)\mathcal{P}_h - S(t)\big)u_{0,1}
    + \big(C_h(t)\mathcal{P}_h - C(t)\big)u_{0,2}.
\end{align*}
Then there exists a constant $C=C(t,\mathcal{O})>0$ such that
\begin{align}
\|F_h(t)X_0\| 
&\le C\, h^{\frac{\beta}{2}}\,|||X_0|||_{\beta}, 
\qquad \beta\in[0,4], \label{main_est5}\\[2mm]
\|G_h(t)X_0\| 
&\le C\, h^{\frac{\beta}{2}}\,|||X_0|||_{\beta}, 
\qquad \beta\in[0,4]. \label{main_est6}
\end{align}
\end{theorem}

	\subsection{\bf Stochastic Version}

We now consider the stochastic linear Schr\"odinger equation \eqref{eqn1} within the 
semigroup framework. As before, we express the equation in real and imaginary 
components. Writing $u=u_1+i u_2$ and $W=(W_1,W_2)^{\mathrm{T}}$, equation 
\eqref{eqn1} is equivalent to the system
\begin{equation}\label{sys_nonhom}
d
\begin{bmatrix}
u_1\\
u_2
\end{bmatrix}
=
\begin{bmatrix}
0 & -\Lambda\\
\Lambda & 0
\end{bmatrix}
\begin{bmatrix}
u_1\\
u_2
\end{bmatrix}\,dt
+
\begin{bmatrix}
dW_1\\
dW_2
\end{bmatrix},
\qquad t>0,
\qquad
\begin{bmatrix}
u_1(0)\\
u_2(0)
\end{bmatrix}
=
\begin{bmatrix}
u_{0,1}\\
u_{0,2}
\end{bmatrix}.
\end{equation}

This may be written abstractly as the stochastic evolution equation
\begin{equation}\label{system1}
dX(t) = A X(t)\,dt + dW(t),\qquad t>0,\qquad X(0)=X_0,
\end{equation}
where
\[
X(t)=
\begin{bmatrix}
u_1(t)\\
u_2(t)
\end{bmatrix},
\qquad 
X_0=
\begin{bmatrix}
u_{0,1}\\
u_{0,2}
\end{bmatrix},
\qquad
dW(t)=
\begin{bmatrix}
dW_1(t)\\
dW_2(t)
\end{bmatrix}.
\]

The mild (weak) solution of \eqref{system1} is given by
\begin{equation}\label{soleqn1}
X(t)
=
E(t)X_0
+
\int_0^t E(t-s)\,dW(s),
\qquad t\ge 0,
\end{equation}
where $E(t)=e^{tA}$ is the semigroup defined in the previous subsection.

It is well known (see, e.g., \cite{pgrisvard}) that the solution satisfies the moment 
estimate
\begin{equation}\label{est1}
\|X(t)\|_{L^2(\Omega;H^\beta)}
\le
C\Big(
|||X_0|||_{L^2(\Omega;H^\beta)}
+
t^{1/2}
\big(
\|\Lambda^{\beta/2}Q_1^{1/2}\|_{HS}
+
\|\Lambda^{\beta/2}Q_2^{1/2}\|_{HS}
\big)
\Big),
\qquad t\ge 0.
\end{equation}

\medskip

Next we derive the finite element approximation of the stochastic system 
\eqref{system1}.  
Using a standard piecewise linear finite element space $V_h$, the spatially discrete 
analogue of \eqref{system1} is to find $X_h(t)=(u_{h,1}(t),u_{h,2}(t))^{\mathrm{T}}
\in V_h\times V_h$ such that
\begin{equation}\label{eqn4}
dX_h(t)
=
A_h X_h(t)\,dt
+
\mathcal{P}_h\, dW(t),
\qquad t>0,
\qquad X_h(0)=X_{0,h},
\end{equation}
where $A_h$ is the discrete operator defined earlier and $\mathcal{P}_h$ denotes 
the $L^2$-projection onto $V_h$.

The unique mild solution of \eqref{eqn4} is
\begin{equation}\label{eqn5}
X_h(t)
=
E_h(t)X_{0,h}
+
\int_0^t E_h(t-s)\mathcal{P}_h\, dW(s),
\qquad t\ge 0.
\end{equation}

\begin{theorem}\label{stoch_err1}
Let $\beta\in[0,4]$ and assume that the covariance operators $Q_i$, $i=1,2$, satisfy
\[
\|\Lambda^{\beta/2}Q_1^{1/2}\|_{HS}
+
\|\Lambda^{\beta/2}Q_2^{1/2}\|_{HS}
<\infty.
\]
Let $X(t)=(u_1(t),u_2(t))^{\mathrm{T}}$ and 
$X_h(t)=(u_{h,1}(t),u_{h,2}(t))^{\mathrm{T}}$ be the solutions of 
\eqref{soleqn1} and \eqref{eqn5}, respectively, with 
\[
X_{0,h}
=
(\mathcal{P}_h u_{0,1},\mathcal{P}_h u_{0,2})^{\mathrm{T}}.
\]
Then, for all $t\ge 0$,
\begin{equation}\label{system8}
\begin{split}
\|u_{h,1}(t)-u_1(t)\|_{L^2(\Omega;\dot{H}^0)}
&\le 
C_t h^{\frac{\beta}{2}}
\Big(
\|X_0\|_{L^2(\Omega;H^\beta)}
+
\|\Lambda^{\beta/2}Q_1^{1/2}\|_{HS}
+
\|\Lambda^{\beta/2}Q_2^{1/2}\|_{HS}
\Big),\\[2mm]
\|u_{h,2}(t)-u_2(t)\|_{L^2(\Omega;\dot{H}^0)}
&\le 
C_t h^{\frac{\beta}{2}}
\Big(
\|X_0\|_{L^2(\Omega;H^\beta)}
+
\|\Lambda^{\beta/2}Q_1^{1/2}\|_{HS}
+
\|\Lambda^{\beta/2}Q_2^{1/2}\|_{HS}
\Big),
\end{split}
\end{equation}
where $C_t$ is a positive function increasing in $t$.
\end{theorem}

The proof of \eqref{system8} relies on the It\^o isometry together with the 
deterministic estimates \eqref{main_est5}--\eqref{main_est6} of Theorem~\ref{det_err2}.
\section{Preliminaries: Notation and Mild Solution Framework}\label{s2}

Let $(U,(\cdot,\cdot)_U)$ and $(H,(\cdot,\cdot)_H)$ be separable Hilbert spaces with 
corresponding norms $\|\cdot\|_U$ and $\|\cdot\|_H$.  
We denote by $\mathcal{L}(U,H)$ the space of bounded linear operators from $U$ to $H$, 
and by $\mathcal{L}_2(U,H)$ the space of Hilbert--Schmidt operators equipped with the 
norm
\[
\|T\|_{\mathcal{L}_2(U,H)}^2 := \sum_{j=1}^{\infty}\|Te_j\|_H^2 < \infty,
\]
where $\{e_j\}_{j=1}^{\infty}$ is an orthonormal basis of $U$.  
In the case $U=H$ we write $\mathcal{L}(U)=\mathcal{L}(U,U)$ and 
$HS:=\mathcal{L}_2(U,U)$.

It is well known that if $S\in\mathcal{L}(U)$ and $T\in\mathcal{L}_2(U,H)$, then 
$TS\in\mathcal{L}_2(U,H)$ and the inequality
\[
\|TS\|_{\mathcal{L}_2(U,H)}
\le 
\|T\|_{\mathcal{L}_2(U,H)}\,\|S\|_{\mathcal{L}(U)}
\]
holds.

\medskip

Let $(\Omega,\mathcal{F},P, \{\mathcal{F}_t\}_{t \geq 0})$ be a filtered probability space.  
We denote by $L^2(\Omega,H)$ the space of $H$-valued square-integrable random 
variables, endowed with the norm
\[
\|v\|_{L^2(\Omega,H)}
:= 
\big(\mathbb{E}\|v\|_H^2\big)^{1/2}
=
\left(\int_{\Omega}\|v(\omega)\|_H^2\,dP(\omega)\right)^{1/2}.
\]

Let $Q\in\mathcal{L}(U)$ be a self-adjoint, positive semidefinite operator with 
finite trace $\operatorname{Tr}(Q)$.  
\begin{definition}

A $U$-valued process $\{W(t)\}_{t\ge 0}$ is called a $Q$-Wiener process with respect 
to a filtration $\{\mathcal{F}_t\}_{t\ge 0}$ if

\begin{enumerate}[(i)]
\item $W(0)=0$ almost surely;
\item $W$ has almost surely continuous trajectories;
\item $W$ has independent increments;
\item $W(t)-W(s)$ is a $U$-valued Gaussian random variable with mean zero and 
covariance operator $(t-s)Q$, for $0\le s\le t$;
\item $W(t)$ is $\mathcal{F}_t$-measurable for each $t\ge 0$;
\item $W(t)-W(s)$ is independent of $\mathcal{F}_s$ for each $0\le s\le t$.
\end{enumerate}
\end{definition}
It is known (see, e.g., \cite{rockner}) that for any $Q$-Wiener process satisfying 
(i)–(iv) one can choose a filtration satisfying the usual conditions so that (v)–(vi) 
also hold.  
Moreover, $W(t)$ possesses the orthogonal expansion
\begin{equation}\label{eqn7}
W(t)
=
\sum_{j=1}^{\infty} \gamma_j^{1/2}\,\beta_j(t)\,e_j,
\end{equation}
where $\{(\gamma_j,e_j)\}_{j=1}^{\infty}$ are the eigenpairs of $Q$ and 
$\{\beta_j\}_{j=1}^{\infty}$ are independent standard real-valued Brownian motions.
The series in \eqref{eqn7} converges in $L^2(\Omega,U)$ since, for $t\ge 0$,
\[
\|W(t)\|_{L^2(\Omega,U)}^2
=
\mathbb{E}\Big\|
\sum_{j=1}^{\infty}\gamma_j^{1/2}e_j\beta_j(t)
\Big\|_U^2
=
t\sum_{j=1}^{\infty}\gamma_j
=
t\,\operatorname{Tr}(Q).
\]

\medskip

We require only the special case of the It\^o integral in which the integrand is 
deterministic.  
If a function $\Phi:[0,\infty)\to \mathcal{L}(U,H)$ is strongly measurable and
\begin{equation}\label{eqn8}
\int_0^t \|\Phi(s)Q^{1/2}\|_{\mathcal{L}_2(U,H)}^2\,ds <\infty,
\end{equation}
then the stochastic integral $\int_0^t\Phi(s)\,dW(s)$ is well defined, and It\^o's 
isometry holds:
\begin{equation}\label{eqn9}
\left\|\int_0^t \Phi(s)\,dW(s)\right\|_{L^2(\Omega,H)}^2
=
\int_0^t 
\|\Phi(s)Q^{1/2}\|_{\mathcal{L}_2(U,H)}^2\,ds.
\end{equation}

More generally, if $Q\in\mathcal{L}(U)$ is a self-adjoint, positive semidefinite 
operator with eigenpairs $\{(\gamma_j,e_j)\}_{j=1}^{\infty}$ but is not trace class, 
i.e., $\operatorname{Tr}(Q)=\infty$, then the series \eqref{eqn7} does not converge 
in $L^2(\Omega,U)$.  
However, it converges in a suitably enlarged Hilbert space, and the stochastic 
integral $\int_0^t \Phi(s)\,dW(s)$ remains well defined provided \eqref{eqn8} holds.  
In this case, $W$ is referred to as a \emph{cylindrical Wiener process} 
(see \cite{prato}).  
An important example is the case $Q=I$, the identity operator.

\medskip

We now consider the abstract stochastic differential equation
\begin{equation}\label{sdeqn1}
dX(t) = A X(t)\,dt + dW(t),\qquad t>0,\qquad X(0)=X_0,
\end{equation}
under the following assumptions:
\begin{enumerate}[(i)]
\item $A:D(A)\subset H\to H$ is the generator of a strongly continuous 
($C_0$-) semigroup $\{E(t)\}_{t\ge 0}$ on $H$;
\item $X_0$ is an $\mathcal{F}_0$-measurable $H$-valued random variable.
\end{enumerate}

\begin{definition}[\cite{prato}, Weak Solution]
An $H$-valued predictable process $\{X(t)\}_{t\ge 0}$ is said to be a weak solution of 
\eqref{sdeqn1} if its trajectories are almost surely Bochner integrable and, for every 
$\eta\in D(A^\ast)$ and all $t\ge 0$,
\begin{equation}
(X(t),\eta)
=
(X_0,\eta)
+
\int_0^t (X(s),A^\ast \eta)\,ds
+
\int_0^t (dW(s),\eta),
\qquad P\text{-a.s.},
\end{equation}
where $(A^\ast,D(A^\ast))$ denotes the adjoint of $(A,D(A))$ in $H$.
\end{definition}

\subsection{Abstract Framework and Regularity}\label{s2.1}

Let $\Lambda=-\Delta$ denote the Laplace operator with domain 
$D(\Lambda)=H^{2}(\mathcal{O})\cap H^{1}_{0}(\mathcal{O})$.  
We set $U=L^{2}(\mathcal{O})$ equipped with the usual inner product $(\cdot,\cdot)$ and 
norm $\|\cdot\|$.  
To describe spatial regularity, we introduce the scale of Hilbert spaces
\[
\dot{H}^{\alpha} := D(\Lambda^{\alpha/2}),\qquad 
\|v\|_{\alpha} := \|\Lambda^{\alpha/2}v\|
             = \Big( \sum_{j=1}^{\infty}\lambda_j^{\alpha}(v,\phi_j)^2 \Big)^{1/2},
\qquad \alpha\in\mathbb{R},\ v\in\dot{H}^{\alpha},
\]
where $\{(\lambda_j,\phi_j)\}_{j=1}^{\infty}$ are the eigenpairs of $\Lambda$ with 
orthonormal eigenfunctions.  
If $\alpha\ge \beta$, then $\dot{H}^{\alpha}\subset \dot{H}^{\beta}$.  
It is well known (see \cite{Thomee}) that
\[
\dot{H}^{0} = U,\qquad 
\dot{H}^{1} = H^{1}_{0}(\mathcal{O}),\qquad 
\dot{H}^{2} = D(\Lambda) = H^{2}(\mathcal{O})\cap H^{1}_{0}(\mathcal{O}),
\]
with equivalent norms, and that for $\beta>0$, the dual space satisfies 
$\dot{H}^{-\beta} = (\dot{H}^{\beta})^{\ast}$.  
The inner product on $\dot{H}^{1}$ is 
\[
(v,w)_{1} := (\nabla v,\nabla w).
\]

\medskip
We also introduce the product spaces
\[
H^{\alpha} := \dot{H}^{\alpha}\times \dot{H}^{\alpha}, \qquad 
|||v|||_{\alpha}^{2} := \|v_{1}\|_{\alpha}^{2} + \|v_{2}\|_{\alpha}^{2},
\qquad \alpha\in\mathbb{R},
\]
and write $H = H^{0} = U\times U$ with norm $|||\cdot|||=|||\cdot|||_0$.

\medskip

For $\alpha\in[-1,0]$, define the operator $A:D(A)\subset H^{\alpha}\to H^{\alpha}$ by
\[
A 
:= 
\begin{bmatrix}
0 & -\Lambda\\[1mm]
\Lambda & 0
\end{bmatrix}, 
\qquad
D(A)
=
\left\{
x=(x_{1},x_{2})^{\mathrm{T}}\in H^{\alpha} 
:
Ax=
\begin{bmatrix}
-\Lambda x_{2}\\[1mm]
\Lambda x_{1}
\end{bmatrix}
\in H^{\alpha}
\right\}
=
H^{\alpha+2},
\]
i.e., $D(A)=\dot{H}^{\alpha+2}\times \dot{H}^{\alpha+2}$.

The operator $A$ generates a unitary $C_{0}$-group $\{E(t)\}_{t\in\mathbb{R}}$ on 
$H^{\alpha}$, given explicitly by
\begin{equation}\label{eqn6}
E(t)
=
\begin{bmatrix}
C(t) & -S(t)\\[1mm]
S(t) & C(t)
\end{bmatrix},
\qquad t\in\mathbb{R},
\end{equation}
where
\[
C(t) := \cos(t\Lambda), \qquad S(t):= \sin(t\Lambda)
\]
are the cosine and sine operators.  
Using the eigenpairs $\{(\lambda_j,\phi_j)\}_{j=1}^{\infty}$ of $\Lambda$, these operators have 
the representations (for $v\in\dot{H}^{\alpha}$ and $t\ge 0$)
\begin{align*}
C(t)v &= \cos(t\Lambda)v 
      = \sum_{j=1}^{\infty}\cos(t\lambda_j)\,(v,\phi_j)\,\phi_j,\\
S(t)v &= \sin(t\Lambda)v
      = \sum_{j=1}^{\infty}\sin(t\lambda_j)\,(v,\phi_j)\,\phi_j.
\end{align*}

	\subsection{Finite Element Approximations}\label{s2.2}

Let $\mathcal{T}_h$ be a regular family of triangulations of $\mathcal{O}$, and let  
$h_K = \operatorname{diam}(K)$ for each element $K\in\mathcal{T}_h$, with  
$h := \max_{K\in\mathcal{T}_h} h_K$.  
We denote by $V_h$ the standard space of continuous, piecewise linear functions on  
$\mathcal{T}_h$ that vanish on $\partial\mathcal{O}$.  
Thus $V_h \subset H_0^{1}(\mathcal{O}) = \dot{H}^{1}$.

The assumption that $\mathcal{O}$ is convex and polygonal ensures that the mesh can be  
fitted exactly to $\partial\mathcal{O}$ and that the elliptic regularity estimate  
\[
\|v\|_{H^{2}(\mathcal{O})} \le C\|\Lambda v\|, \qquad v\in D(\Lambda),
\]
holds (see \cite{pgrisvard}).  
We recall several standard results from the finite element theory  
\cite{Brenner,ciarlet}.  
Throughout this subsection, norms of the form $\|\cdot\|_{s}$ refer to  
$\dot{H}^{s}$-norms.

\medskip

Let  
\[
\mathcal{P}_h : \dot{H}^{0} \to V_h,
\qquad
\mathcal{R}_h : \dot{H}^{1} \to V_h,
\]
denote the $L^{2}$-orthogonal projection and the Ritz projection, respectively, defined by
\[
(\mathcal{P}_h v,\chi) = (v,\chi), \qquad
(\nabla \mathcal{R}_h v, \nabla \chi) = (\nabla v, \nabla \chi),
\qquad \forall\, \chi\in V_h.
\]
Then the following approximation properties hold:
\begin{align}
\|(\mathcal{R}_h - I)v\|_{r}
&\le C h^{\,s-r} \|v\|_{s}, 
\qquad r=0,1,\ \ s=1,2,\ \ v\in\dot{H}^{s}, \label{eqn10} \\[1mm]
\|(\mathcal{P}_h - I)v\|_{r}
&\le C h^{\,s-r} \|v\|_{s}, 
\qquad r=-1,0,\ \ s=1,2,\ \ v\in\dot{H}^{s}. \label{eqn11}
\end{align}

\medskip

We now introduce the discrete norms corresponding to $\|\cdot\|_{\alpha}$:
\[
\|v_h\|_{h,\alpha} := \|\Lambda_h^{\alpha/2} v_h\|,
\qquad v_h\in V_h,\quad \alpha\in\mathbb{R},
\]
where $\Lambda_h : V_h \to V_h$ is the discrete Laplacian defined by
\[
(\Lambda_h v_h,\chi) = (\nabla v_h,\nabla \chi),
\qquad \forall\, \chi\in V_h.
\]

\section{Proofs of the Main Results}\label{prf_main}

In this section, we provide the proofs of the results stated in Section~\ref{main_results}.

\begin{proof}[\bf Proof of Theorem \ref{est_appr_1}]
We take the test functions 
\[
\chi_i = \Lambda_h^{\alpha} u_{h,i}(t), \qquad i=1,2,
\]
in \eqref{system5}. Then, for $t \ge 0$,
\[
\begin{split}
(\dot u_{h,1}(t), \Lambda_h^\alpha u_{h,1}(t))
+ (\nabla u_{h,2}(t), \nabla \Lambda_h^\alpha u_{h,1}(t))
&= (f_1(t), \Lambda_h^\alpha u_{h,1}(t)),\\
(\dot u_{h,2}(t), \Lambda_h^\alpha u_{h,2}(t))
- (\nabla u_{h,1}(t), \nabla \Lambda_h^\alpha u_{h,2}(t))
&= (f_2(t), \Lambda_h^\alpha u_{h,2}(t)).
\end{split}
\]

Adding the two identities and using the definition of $\mathcal{P}_h$ yields
\[
\begin{aligned}
(\dot u_{h,1}(t), \Lambda_h^\alpha u_{h,1}(t))
+(\dot u_{h,2}(t), \Lambda_h^\alpha u_{h,2}(t))
&=
(\mathcal{P}_h f_1(t), \Lambda_h^\alpha u_{h,1}(t))
+
(\mathcal{P}_h f_2(t), \Lambda_h^\alpha u_{h,2}(t)).
\end{aligned}
\]

Here we used the symmetry of $\Lambda_h$, which implies
\begin{align*}
	(\nabla u_{h,2}(t),\nabla \Lambda_h^{\alpha}u_{h,1}(t))
	& =(\Lambda_h u_{h,2}(t), \Lambda_h^{\alpha}u_{h,1}(t)),\\
	& =(\Lambda_h^\alpha u_{h,2}(t), \Lambda_h u_{h,1}(t)) ,\\
	& =(\Lambda_h u_{h,1}(t), \Lambda_h^\alpha u_{h,2}(t)),\\
	&= (\nabla u_{h,1}(t),\nabla \Lambda_h^{\alpha}u_{h,2}(t)).
\end{align*}
Using the symmetric property of the operator $\Lambda_h$ and Cauchy-Schwarz inequality,  we get
\[
\begin{split}
    &\frac{d}{dt}\left[\|\Lambda_h^{\alpha/2}u_{h,1}(t)\|^2 + \|\Lambda_h^{\alpha/2}u_{h,2}(t)\|^2\right] \\
    &\qquad=2(\Lambda_h^{\alpha/2}\mathcal{P}_hf_1(t),\Lambda_h^{\alpha/2}u_{h,1}(t))+2(\Lambda_h^{\alpha/2} \mathcal{P}_hf_2(t),\Lambda_h^{\alpha/2}u_{h,2}(t)),\\
    & \qquad\leq 2\|\Lambda_h^{\alpha/2}\mathcal{P}_hf_1(t)\|\|\Lambda_h^{\alpha/2}u_{h,1}(t)\|+2\|\Lambda_h^{\alpha/2}\mathcal{P}_hf_2(t)\|\|\Lambda_h^{\alpha/2}u_{h,2}(t)\|.
\end{split}
\]
Using Cauchy-Schwarz inequality (in $\mathbb{R}^2$), we can write
\[
\begin{split}
	& \frac{d}{dt}\Big(\|u_{h,1}(t)\|^2_{h,\alpha}+\|u_{h,1}(t)\|^2_{h,\alpha}\Big) \\
	& \qquad\leq 2\Big(\|\mathcal{P}_hf_1(t)\|_{h,\alpha}\|u_{h,1}(t)\|_{h,\alpha} 
	+\|\mathcal{P}_hf_2(t)\|_{h,\alpha}\|u_{h,2}(t)\|_{h,\alpha}\Big),\\
	& \qquad\leq 2\Big(\|\mathcal{P}_hf_1(t)\|_{h,\alpha}^2 
	+\|\mathcal{P}_hf_2(t)\|_{h,\alpha}^2\Big)^\frac{1}{2} \Big(\|u_{h,1}(t)\|^2_{h,\alpha}+\|u_{h,1}(t)\|^2_{h,\alpha}\Big)^\frac{1}{2},\\
	& \qquad\leq 2\Big(\|\mathcal{P}_hf_1(t)\|_{h,\alpha} 
	+\|\mathcal{P}_hf_2(t)\|_{h,\alpha}\Big) \Big(\|u_{h,1}(t)\|^2_{h,\alpha}+\|u_{h,1}(t)\|^2_{h,\alpha}\Big)^\frac{1}{2}.
\end{split}
\]

Define
\[
g(t) := \Big(\|u_{h,1}(t)\|_{h,\alpha}^2 + \|u_{h,2}(t)\|_{h,\alpha}^2\Big)^{1/2}.
\]
Then the previous inequality becomes
\[
\frac{d}{dt} g^2(t)
\le
2\big(\|\mathcal{P}_h f_1(t)\|_{h,\alpha}
+
\|\mathcal{P}_h f_2(t)\|_{h,\alpha}\big) g(t),
\]
which implies
\[
\frac{dg}{dt}(t)
\le
\|\mathcal{P}_h f_1(t)\|_{h,\alpha}
+
\|\mathcal{P}_h f_2(t)\|_{h,\alpha}.
\]

Integrating over $[0,t]$ gives
\[
g(t)
\le
g(0)
+
\int_0^{t}
\big(\|\mathcal{P}_h f_1(s)\|_{h,\alpha}
+
\|\mathcal{P}_h f_2(s)\|_{h,\alpha}\big)\,ds,
\]
where
\[
g(0) = \Big(\|u_{h,0,1}\|_{h,\alpha}^2 + \|u_{h,0,2}\|_{h,\alpha}^2\Big)^{1/2}.
\]

Therefore, for all $t \ge 0$,
\[
\begin{split}
\|u_{h,1}(t)\|_{h,\alpha} + \|u_{h,2}(t)\|_{h,\alpha}
\le
C\Big(
\|u_{h,0,1}\|_{h,\alpha}
+
\|u_{h,0,2}\|_{h,\alpha}
+
\int_0^t
(\|\mathcal{P}_h f_1(s)\|_{h,\alpha}
+
\|\mathcal{P}_h f_2(s)\|_{h,\alpha})\,ds
\Big),
\end{split}
\]
which completes the proof.
\end{proof}
We now establish the error estimate for the deterministic non-homogeneous problem.
\begin{proof}[\bf Proof of Theorem \ref{det_err1}]
We decompose the error as
\[
e_i := u_{h,i} - u_i = (u_{h,i} - \mathcal{R}_h u_i) + (\mathcal{R}_h u_i - u_i)
=:\theta_i + \rho_i, \qquad i=1,2.
\]

Subtracting \eqref{system4} from \eqref{system5}, we obtain for all 
$\chi_1,\chi_2 \in V_h$,
\[
\begin{split}
(\dot{u}_{h,1}(t) - \dot{u}_1(t), \chi_1)
+ (\nabla (u_{h,2}(t)-u_2(t)), \nabla \chi_1) &= 0,\\
(\dot{u}_{h,2}(t) - \dot{u}_2(t), \chi_2)
-(\nabla (u_{h,1}(t)-u_1(t)), \nabla \chi_2) &= 0,
\qquad t>0.
\end{split}
\]

With $e_i = \theta_i + \rho_i$, this becomes
\[
\begin{split}
(\dot{e}_1(t),\chi_1) + (\nabla e_2(t),\nabla \chi_1) &= 0,\\
(\dot{e}_2(t),\chi_2) - (\nabla e_1(t),\nabla \chi_2) &= 0.
\end{split}
\]

Expanding using $\theta_i$ and $\rho_i$ yields
\[
\begin{split}
(\dot{\theta}_1(t),\chi_1) + (\nabla \theta_2(t),\nabla\chi_1)
&= -(\dot{\rho}_1(t),\chi_1) - (\nabla \rho_2(t),\nabla\chi_1),\\
(\dot{\theta}_2(t),\chi_2) - (\nabla \theta_1(t),\nabla\chi_2)
&= (\nabla \rho_1(t),\nabla\chi_2) - (\dot{\rho}_2(t),\chi_2).
\end{split}
\]

By definition of the Ritz projector $\mathcal{R}_h$,  
\[
(\nabla \rho_2, \nabla \chi_1) = 0,
\qquad 
(\nabla \rho_1, \nabla \chi_2) = 0,
\]
so that
\begin{equation}\label{theta-eq}
\begin{split}
(\dot{\theta}_1(t),\chi_1) + (\nabla \theta_2(t),\nabla\chi_1)
&= -(\dot{\rho}_1(t),\chi_1),\\
(\dot{\theta}_2(t),\chi_2) - (\nabla \theta_1(t),\nabla\chi_2)
&= -(\dot{\rho}_2(t),\chi_2).
\end{split}
\end{equation}

Thus $(\theta_1,\theta_2)$ solves \eqref{system5} with  
$f_1 = -\dot{\rho}_1$ and $f_2 = -\dot{\rho}_2$.
Applying estimate \eqref{eqn2} with $\alpha=0$, we obtain
\begin{equation}\label{theta-est}
\begin{split}
\|\theta_1(t)\|_{h,0} + \|\theta_2(t)\|_{h,0}
\le C\Big(
&\|\theta_1(0)\|_{h,0} + \|\theta_2(0)\|_{h,0} \\
&+ \int_0^t \|\mathcal{P}_h\dot{\rho}_1(s)\|_{h,0}\,ds
+ \int_0^t \|\mathcal{P}_h\dot{\rho}_2(s)\|_{h,0}\,ds
\Big).
\end{split}
\end{equation}

Since $\theta_i(t)\in V_h$, we have $\|\theta_i(t)\|=\|\theta_i(t)\|_{h,0}$.
Therefore,
\[
\|e_1(t)\|
\le \|\theta_1(t)\|_{h,0} + \|\rho_1(t)\|.
\]
Using \eqref{theta-est},
\begin{equation}\label{err1-temp}
\begin{split}
\|e_1(t)\|
\le C\Big(
&\|\theta_1(0)\|_{h,0} + \|\theta_2(0)\|_{h,0}
+ \int_0^t \|\mathcal{P}_h\dot{\rho}_1(s)\|\,ds
+ \int_0^t \|\mathcal{P}_h\dot{\rho}_2(s)\|\,ds
+ \|\rho_1(t)\|
\Big).
\end{split}
\end{equation}

Using \eqref{eqn10} with $r=0$, $s=2$, we obtain for $i=1,2$,
\[
\|(\mathcal{R}_h - I)\dot u_i(s)\| \le C h^2 \|\dot u_i(s)\|_2,
\qquad
\|(\mathcal{R}_h - I)u_i(t)\| \le C h^2 \|u_i(t)\|_2.
\]

Also,
\[
\|\theta_i(0)\|_{h,0}
= \|u_{h,0,i} - \mathcal{R}_h u_{0,i}\|.
\]

Substituting into \eqref{err1-temp} yields
\[
\begin{split}
\|e_1(t)\|
\le C\Big(
&\|u_{h,0,1} - \mathcal{R}_h u_{0,1}\|
+ \|u_{h,0,2} - \mathcal{R}_h u_{0,2}\| \\
&+ h^2 \int_0^t \|\dot u_1(s)\|_2\,ds
+ h^2 \int_0^t \|\dot u_2(s)\|_2\,ds
+ h^2 \|u_1(t)\|_2
\Big).
\end{split}
\]

The estimate for $e_2(t)$ is analogous, leading to
\[
\begin{split}
\|e_2(t)\|
\le C\Big(
&\|u_{h,0,1} - \mathcal{R}_h u_{0,1}\|
+ \|u_{h,0,2} - \mathcal{R}_h u_{0,2}\| \\
&+ h^2 \int_0^t \|\dot u_1(s)\|_2\,ds
+ h^2 \int_0^t \|\dot u_2(s)\|_2\,ds
+ h^2 \|u_2(t)\|_2
\Big).
\end{split}
\]

This completes the proof.
\end{proof}

	We now establish the error bounds \eqref{main_est5}–\eqref{main_est6} for the 
finite element approximation of the deterministic homogeneous linear Schrödinger 
equation.

\begin{proof}[\bf Proof of Theorem \ref{det_err2}]
We first consider the case $\beta = 0$.  
Recall that
\[
F_h(t)X_0 = u_{h,1}(t) - u_1(t), 
\qquad 
G_h(t)X_0 = u_{h,2}(t) - u_2(t),
\]
where $(u_1,u_2)^{\mathrm{T}}$ solves \eqref{abs_hom} with initial data 
$X_0=(u_{0,1},u_{0,2})^{\mathrm{T}}$, and $(u_{h,1},u_{h,2})^{\mathrm{T}}$ solves \eqref{homapp_eqn} 
with initial values 
$(u_{h,0,1},u_{h,0,2})^{\mathrm{T}} = (\mathcal{P}_h u_{0,1}, \mathcal{P}_h u_{0,2})^{\mathrm{T}}$.

Using the energy identity \eqref{eqn3} for the continuous solution (with $\alpha=0$) 
and the stability estimate \eqref{eqn2} for the discrete system, we obtain
\[
\begin{aligned}
\|F_h(t)X_0\|
&\le \|u_{h,1}(t)\| + \|u_1(t)\| \\
&\le C\left(
\|u_{h,0,1}\| + \|u_{h,0,2}\| 
+ \|u_{0,1}\| + \|u_{0,2}\|
\right)\\
&= C\left(
\|\mathcal{P}_h u_{0,1}\| + \|\mathcal{P}_h u_{0,2}\| 
+ \|u_{0,1}\| + \|u_{0,2}\|
\right)\\
&\le C(|||X_0|||_0).
\end{aligned}
\]
This proves \eqref{main_est5} for $\beta=0$.  
The same argument applies to $G_h(t)X_0$.

To estimate the error for $\beta=4$, we use Theorem~\ref{det_err1} with initial data  
$(u_{h,0,1},u_{h,0,2}) = (\mathcal{P}_h u_{0,1}, \mathcal{P}_h u_{0,2})$.
Then
\[
F_h(t)X_0 = e_1(t) = u_{h,1}(t) - u_1(t).
\]
Using estimate \eqref{main_est3},
\[
\begin{aligned}
\|F_h(t)X_0\|
\le {}
&\|\mathcal{P}_h(I-\mathcal{R}_h)u_{0,1}\|
+ \|\mathcal{P}_h(I-\mathcal{R}_h)u_{0,2}\| \\
&\quad
+ C h^2\left(
\int_0^t \|\dot u_1(s)\|_2\,ds
+ \int_0^t \|\dot u_2(s)\|_2\,ds
+ \|u_1(t)\|_2
\right).
\end{aligned}
\]

Using the projection estimate \eqref{eqn10} and the homogeneous energy identity 
\eqref{eqn3}, we have
\[
\|(\mathcal{R}_h - I)u_{0,i}\| \le Ch^2 \|u_{0,i}\|_2,\quad \text{ for }i=1,2.
\]
Therefore,
\[
\begin{aligned}
\|F_h(t)X_0\|
&\le Ch^2\left(
\|u_{0,1}\|_2 + \|u_{0,2}\|_2
+ t(\|u_{0,1}\|_4 + \|u_{0,2}\|_4)
\right)\\
&\le Ch^2(\|u_{0,1}\|_4 + \|u_{0,2}\|_4)
= Ch^2 |||X_0|||_4.
\end{aligned}
\]

By interpolation between the cases $\beta=0$ and $\beta=4$, we obtain
\[
\|F_h(t)X_0\| \le Ch^{\frac{\beta}{2}} |||X_0|||_\beta,
\qquad 0 \le \beta \le 4.
\]

The proof for $G_h(t)X_0$ is identical, which establishes 
\eqref{main_est5}–\eqref{main_est6}.
\end{proof}

	We now prove Theorem \ref{stoch_err1}, which gives the strong error estimates for 
the finite element approximation of the stochastic linear Schrödinger equation.

\begin{proof}[\bf Proof of Theorem \ref{stoch_err1}]
It suffices to prove only the first inequality in \eqref{system8}; the second 
inequality follows in the same way.

The first component $u_1(t)$ of the mild solution \eqref{soleqn1} of 
\eqref{system1} is
\[
u_1(t)
= C(t)u_{0,1} - S(t)u_{0,2}
+ \int_0^t C(t-s)\,dW_1(s)
- \int_0^t S(t-s)\,dW_2(s), 
\qquad t \ge 0.
\]

Similarly, the first component $u_{h,1}(t)$ of \eqref{eqn5}, the solution of 
\eqref{eqn4}, is
\[
u_{h,1}(t)
= C_h(t)u_{h,0,1} -S_h(t)u_{h,0,2}
+ \int_0^t C_h(t-s)\mathcal{P}_h\,dW_1(s)
- \int_0^t S_h(t-s)\mathcal{P}_h\,dW_2(s),
\qquad t\ge0.
\]

Since $u_{h,0,1}=\mathcal{P}_h u_{0,1}$ and $u_{h,0,2}=\mathcal{P}_h u_{0,2}$, we obtain
\[
\begin{aligned}
u_{h,1}(t)-u_1(t)
&= \big(C_h(t)\mathcal{P}_h - C(t)\big)u_{0,1}
   -\big(S_h(t)\mathcal{P}_h - S(t)\big)u_{0,2} \\
&\qquad + \int_0^t \big(C_h(t-s)\mathcal{P}_h - C(t-s)\big)\,dW_1(s) \\
&\qquad - \int_0^t \big(S_h(t-s)\mathcal{P}_h - S(t-s)\big)\,dW_2(s).
\end{aligned}
\]

Using the definition of $F_h(t)$ from Theorem \ref{det_err2}, we may write
\[
\begin{aligned}
u_{h,1}(t)-u_1(t)
= F_h(t)X_0
&+ \int_0^t \big(C_h(t-s)\mathcal{P}_h - C(t-s)\big)\,dW_1(s) \\
&- \int_0^t \big(S_h(t-s)\mathcal{P}_h - S(t-s)\big)\,dW_2(s).
\end{aligned}
\]

Taking the $L^2(\Omega;\dot H^0)$–norm, we obtain
\[
\begin{aligned}
\|u_{h,1}(t)-u_1(t)\|_{L^2(\Omega;\dot H^0)}
&\le I_1(t) + I_2(t) + I_3(t),
\end{aligned}
\]
where
\[
\begin{aligned}
I_1(t) &= \|F_h(t)X_0\|_{L^2(\Omega;\dot H^0)},\\
I_2(t) &= \left\|\int_0^t (C_h(t-s)\mathcal{P}_h - C(t-s))\,dW_1(s)\right\|_{L^2(\Omega;\dot H^0)},\\
I_3(t) &= \left\|\int_0^t (S_h(t-s)\mathcal{P}_h - S(t-s))\,dW_2(s)\right\|_{L^2(\Omega;\dot H^0)}.
\end{aligned}
\]

From \eqref{main_est5},
\[
I_1^2(t)
= \mathbb{E}\|F_h(t)X_0\|^2
\le C_t^2 h^{\beta} \,\mathbb{E}|||X_0|||_{\beta}^2.
\]

By Itô’s isometry,
\[
I_2^2(t)
= \int_0^t 
\|(C_h(t-s)\mathcal{P}_h - C(t-s))Q_1^{1/2}\|_{HS}^2\,ds.
\]

Introduce the operator
\[
K_h(t)g = (C_h(t)\mathcal{P}_h - C(t))g, \qquad g\in \dot H^\beta.
\]

From \eqref{main_est5} (with $u_{0,2}=0$),
\[
\|K_h(t)g\| \le C h^{\beta/2} \|g\|_\beta.
\]

Using the orthonormal basis $\{e_i\}$ of $L^2(\mathcal O)$,
\[
\begin{aligned}
I_2^2(t)
&= \int_0^t \sum_{i=1}^\infty 
   \|K_h(t-s)Q_1^{1/2}e_i\|^2\,ds \\
&\le Ch^\beta t 
   \sum_{i=1}^\infty \|\Lambda^{\beta/2}Q_1^{1/2}e_i\|^2 \\
&\le Ch^\beta t\|\Lambda^{\beta/2}Q_1^{1/2}\|_{HS}^2.
\end{aligned}
\]

Similarly,
\[
I_3^2(t)
= \int_0^t 
\|(S_h(t-s)\mathcal{P}_h - S(t-s))Q_2^{1/2}\|_{HS}^2\,ds.
\]

Introduce
\[
L_h(t)g = (S_h(t)\mathcal{P}_h - S(t))g, \qquad g\in \dot H^\beta.
\]

From \eqref{main_est6} (with $u_{0,2}=0$),
\[
\|L_h(t)g\| \le C h^{\beta/2} \|g\|_\beta.
\]

Hence,
\[
I_3^2(t)
\le Ch^\beta t \|\Lambda^{\beta/2}Q_2^{1/2}\|_{HS}^2.
\]

Combining the bounds on $I_1(t)$, $I_2(t)$, and $I_3(t)$, we obtain
\[
\|u_{h,1}(t)-u_1(t)\|_{L^2(\Omega;\dot H^0)}
\le C_t h^{\frac{\beta}{2}} 
\left(
|||X_0|||_{L^2(\Omega;H^\beta)}
+ \|\Lambda^{\beta/2}Q_1^{1/2}\|_{HS}
+ \|\Lambda^{\beta/2}Q_2^{1/2}\|_{HS}
\right),
\]
which proves the first estimate of \eqref{system8}.  
The second estimate follows analogously.

This completes the proof.
\end{proof}

	\section{Numerical Experiments}\label{S_4}

In this section, we present numerical computations supporting the theoretical strong convergence results established in Theorem~\ref{stoch_err1} for the stochastic linear Schr\"{o}dinger equation. We illustrate the behaviour of the finite element approximation through representative numerical examples. For the temporal discretization, we employ the implicit backward Euler method.

\subsection{Computational Analysis}

We consider the following system for $T>0$:
\[
\begin{bmatrix}
du_{h,1}(t)\\[2mm]
du_{h,2}(t)
\end{bmatrix}
=
\begin{bmatrix}
0 & -\Lambda_h\\
\Lambda_h & 0
\end{bmatrix}
\begin{bmatrix}
u_{h,1}(t)\\[2mm]
u_{h,2}(t)
\end{bmatrix}\,dt
+
\begin{bmatrix}
\mathcal{P}_h dW_1(t)\\[2mm]
\mathcal{P}_h dW_2(t)
\end{bmatrix},
\qquad t\in[0,T].
\]

Let $P_N=\{0=t_0<t_1<\cdots<t_N=T\}$ be a uniform partition of $[0,T]$ with step size $k=T/N$, and denote $I_n=(t_{n-1},t_n)$ for $n=1,\ldots,N$.  
The backward Euler time discretization then reads
\[
\begin{bmatrix}
U^n_1 \\
U^n_2
\end{bmatrix}
-
\begin{bmatrix}
U^{n-1}_1 \\
U^{n-1}_2
\end{bmatrix}
=
\begin{bmatrix}
0 & -k\Lambda_h\\
k\Lambda_h & 0
\end{bmatrix}
\begin{bmatrix}
U^n_1\\
U^n_2
\end{bmatrix}
+
\begin{bmatrix}
\mathcal{P}_h\Delta W_1^n\\
\mathcal{P}_h\Delta W_2^n
\end{bmatrix},
\]
where $U_i^n\in V_h$ approximates $u_i(\cdot,t_n)$ for $i=1,2$ and $n=1,\ldots,N$.

Rewriting the scheme in a more convenient form yields
\begin{equation}\label{time_dis}
\begin{bmatrix}
I & k\Lambda_h\\
-k\Lambda_h & I
\end{bmatrix}
\begin{bmatrix}
U_1^{n}\\[1mm]
U_2^{n}
\end{bmatrix}
=
\begin{bmatrix}
U_1^{n-1}\\[1mm]
U_2^{n-1}
\end{bmatrix}
+
\begin{bmatrix}
\mathcal{P}_h\Delta W_1^n\\[1mm]
\mathcal{P}_h\Delta W_2^n
\end{bmatrix}.
\end{equation}
We approximate the noise by a truncated Fourier expansion.  
For the noises $W_i$, $i=1,2$, we have for all $\chi\in V_h$,
\begin{equation} \label{system10}
    (\mathcal{P}_h \Delta W_i^n, \chi)
    = \sum_{j=1}^{\infty} \gamma_{j,i}^{1/2}\,\Delta \beta^n_{j,i}(e_j,\chi)
    \;\approx\;
    \sum_{j=1}^{J} \gamma_{j,i}^{1/2}\,\Delta \beta^n_{j,i}(e_j,\chi),
\end{equation}
where the infinite series is truncated after $J$ terms.  
Here $\{\beta_{j,i}(t)\}_{j=1}^J$ denote mutually independent standard real-valued Brownian motions for each $i=1,2$.  
The increments in \eqref{system10} satisfy
\[
    \Delta\beta^n_{j,i}
    = \beta_{j,i}(t_n)-\beta_{j,i}(t_{n-1})
    \sim \sqrt{k}\,\mathcal{N}(0,1),
\]
where $\mathcal{N}(0,1)$ denotes a Gaussian random variable with mean zero and unit variance.  
For space–time white noise, we also have $\gamma_{j,i}=1$.

We denote by $X_h^J=(u_{h,1}^J,u_{h,2}^J)^{\mathrm{T}}$ the semidiscrete solution obtained with the truncated noise, that is,
\begin{equation} \label{system11}
    X_h^J(t)
    =E_h(t)X_{h,0}
    +\sum_{j=1}^{J}\int_{0}^{t}E_h(t-s)\mathcal{P}_h e_j\,d\beta_j(s),
    \qquad t\in[0,T],
\end{equation}
where 
\[
    \beta_j(s)
    =\big(\gamma_{j,1}^{1/2}\beta_{j,1}(s),\, \gamma_{j,2}^{1/2}\beta_{j,2}(s)\big)^{\mathrm{T}} .
\]

\begin{lemma}
Let $X_h^J$ and $X_h$ be defined by \eqref{system11} and \eqref{eqn5}, respectively.  
Assume that $\Lambda$, $Q_1$, and $Q_2$ share a common orthonormal eigenbasis $\{e_j\}_{j=1}^{\infty}$, and that $V_h$ (of dimension $N_h$) is constructed on a family of quasi-uniform triangulations $\{\mathcal{T}_h\}$ of $\mathcal{O}$.  
If $J\ge N_h$ and
\[
    \|\Lambda^{\beta/2}Q_1^{1/2}\|_{HS}
    +\|\Lambda^{\beta/2}Q_2^{1/2}\|_{HS}<\infty,
    \qquad \beta\in[0,4],
\]
then there exists $C=C_{t,J}>0$ such that, for all $t\ge 0$,
\begin{equation}\label{numest2}
\begin{split}
    \|u_{h,1}^J(t)-u_{h,1}(t)\|_{L^2(\Omega;\dot{H}^0)}
    &\le Ch^{\beta/2}
        \big(
            \|\Lambda^{\beta/2}Q_1^{1/2}\|_{HS}
            +\|\Lambda^{\beta/2}Q_2^{1/2}\|_{HS}
        \big),\\
    \|u_{h,2}^J(t)-u_{h,2}(t)\|_{L^2(\Omega;\dot{H}^0)}
    &\le Ch^{\beta/2}
        \big(
            \|\Lambda^{\beta/2}Q_1^{1/2}\|_{HS}
            +\|\Lambda^{\beta/2}Q_2^{1/2}\|_{HS}
        \big).
\end{split}
\end{equation}
\end{lemma}

\begin{proof}
We prove only the estimate for the first component; the second follows analogously.  
From \eqref{system11} and \eqref{eqn5}, for $t\ge 0$,
\[
\begin{split}
u_{h,1}(t)-u_{h,1}^J(t)
    &=\sum_{j=J+1}^{\infty}\gamma_{j,1}^{1/2}
        \int_{0}^{t}C_h(t-s)\mathcal{P}_h e_j\,d\beta_{j,1}(s) \\
    &\qquad
        -\sum_{j=J+1}^{\infty}\gamma_{j,2}^{1/2}
        \int_{0}^{t}S_h(t-s)\mathcal{P}_h e_j\,d\beta_{j,2}(s).
\end{split}
\]

Applying It\^o's isometry, independence of the Brownian motions, and the error properties of $C_h$ and $S_h$, we obtain
\begin{equation}\label{numest1}
\begin{split}
&\|u_{h,1}(t)-u_{h,1}^J(t)\|_{L^2(\Omega;\dot{H}^0)}^{2}  \\
&\quad\le
2\sum_{j=J+1}^{\infty}\gamma_{j,1}
    \int_{0}^{t}\|C_h(s)\mathcal{P}_h e_j\|^{2}\,ds
+2\sum_{j=J+1}^{\infty}\gamma_{j,2}
    \int_{0}^{t}\|S_h(s)\mathcal{P}_h e_j\|^{2}\,ds \\
&\quad=
2\sum_{j=J+1}^{\infty}\gamma_{j,1}
    \int_{0}^{t}\|(C_h(s)\mathcal{P}_h-C(s))e_j+C(s)e_j\|^{2}\,ds \\
&\qquad
+2\sum_{j=J+1}^{\infty}\gamma_{j,2}
    \int_{0}^{t}\|(S_h(s)\mathcal{P}_h-S(s))e_j+S(s)e_j\|^{2}\,ds\\
&\quad\le
4\sum_{j=J+1}^{\infty}\gamma_{j,1}
    \int_{0}^{t}\|(C_h(s)\mathcal{P}_h-C(s))e_j\|^{2}ds
+4\sum_{j=J+1}^{\infty}\gamma_{j,1}
    \int_{0}^{t}\|C(s)e_j\|^{2}ds\\
&\qquad
+4\sum_{j=J+1}^{\infty}\gamma_{j,2}
    \int_{0}^{t}\|(S_h(s)\mathcal{P}_h-S(s))e_j\|^{2}ds
+4\sum_{j=J+1}^{\infty}\gamma_{j,2}
    \int_{0}^{t}\|S(s)e_j\|^{2}ds \\
&\quad:= I_1(t)+I_2(t)+I_3(t)+I_4(t).
\end{split}
\end{equation}

Using \eqref{main_est5} with $X_0=(e_j,0)^{\mathrm{T}}$, we get  
\[
\begin{split}
I_1(t)
    &=4\sum_{j=J+1}^{\infty}\gamma_{j,1}
        \int_{0}^{t}\|K_h(s)e_j\|^{2}ds \\
    &\le 4C h^{\beta} t
        \sum_{j=J+1}^{\infty}\|\Lambda^{\beta/2}Q_1^{1/2}e_j\|^{2}
    \le 4Ch^{\beta}t\|\Lambda^{\beta/2}Q_1^{1/2}\|_{HS}^{2}.
\end{split}
\]

Similarly,
\[
I_3(t)
    \le 4Ch^{\beta}t\|\Lambda^{\beta/2}Q_2^{1/2}\|_{HS}^{2}.
\]

For $I_2(t)$,
\[
\begin{split}
I_2(t)
    &=4\sum_{j=J+1}^{\infty}\gamma_{j,1}
        \int_{0}^{t}\cos^{2}(s\lambda_j)\,ds \\
    &\le 4t\sum_{j=J+1}^{\infty}\gamma_{j,1}
        =4t\sum_{j=J+1}^{\infty}\lambda_j^{-\beta/2}(\lambda_j^{\beta/2}\gamma_{j,1})\\
    &\le 4t \lambda_{J+1}^{-\beta/2}
        \|\Lambda^{\beta/4}Q_1^{1/2}\|_{HS}^{2} \leq  4t \lambda_{J+1}^{-\beta/2}
        \|\Lambda^{\beta/2}Q_1^{1/2}\|_{HS}^{2}.
\end{split}
\]

Similarly,
\[
I_4(t)
    \le 4t \lambda_{J+1}^{-\beta/2}
        \|\Lambda^{\beta/2}Q_2^{1/2}\|_{HS}^{2}.
\]

Finally, since for quasi-uniform meshes,
\[
N_h \approx h^{-d},
\qquad \lambda_j \approx j^{2/d},
\]
we have
\[
\lambda_{J+1}^{-1}
    \le CJ^{-2/d}
    \le C N_h^{-2/d}
    \le Ch^{2}.
\]

Combining the four terms $I_1$–$I_4$ in \eqref{numest1} gives \eqref{numest2}, completing the proof.
\end{proof}
\begin{remark}
\begin{enumerate}[(i)]
    \item The above lemma shows that, under appropriate assumptions on the triangulation and on the covariance operators $Q_i$, $i=1,2$, it is sufficient to choose $J \ge N_h$, where $N_h=\mathrm{dim}(V_h)$.  
    In this case, the convergence order of the finite element method is preserved.

    \item In general, the operators $\Lambda$ and $\{Q_i: i=1,2\}$ may not possess a common orthonormal basis of eigenfunctions.  
    In practical computations, the eigenfunctions of $\{Q_i: i=1,2\}$ are typically not known in closed form.  
    To represent $\mathcal{P}_h W_i$, one would need to solve the discrete eigenvalue problem  
    \[
        Q_i \phi = \lambda \phi \qquad \text{in } V_h,
    \]
    which becomes computationally expensive when $Q_i$ is given through an integral kernel.  
    If the kernel is sufficiently smooth, more efficient techniques can be used; see, for example, \cite{Cschwab}.
\end{enumerate}
\end{remark}

	\subsection{Numerical Example}
	We consider the following stochastic linear Schr\"{o}dinger equation in one spatial dimension.
	\begin{equation}\label{eqn_num1}
		\begin{split}
			&du+\,i \Delta u\,dt =dW_1 +i\, dW_2  \quad \text{in}\quad ( 0,1)\times  ( 0,1), \\
			&u(t,0)=0=u(t,1), \quad t\in (0,1),\\
			&u(0,x)=\sin(2\pi x)+i\,x(1-x), \quad x\in ( 0,1).
		\end{split}
	\end{equation}
	To find a numerical error, we consider that numerical solution with a very finer mesh (say $h_{\text{ref}}$) to be exact. We find the approximate value of $u(x,1)=u_1(x,1)+i\,u_2(x,1),$ using the implicit Euler method for time discretization with a very small fixed time step $k.$ Applying the time stepping \eqref{time_dis} to the system \eqref{eqn_num1} and considering the finite element approximation $V_h$ with mass matrix $M$, we obtain the discrete system 
	\begin{equation}\label{time_dis2}
		(M+kL_h)X^n=MX^{n-1}+B.
	\end{equation}
	We note that for a deterministic system i.e. $B=0$, the expected rate of convergence for both real and imaginary components in the $L^2$ norm is $2$ (see \eqref{main_est3} and \eqref{main_est4}).
	
	Let $\{\lambda_j\}_{j=1}^\infty$ be eigen values of $\Lambda$ and we take $Q_1=Q_2=\Lambda^{-s},\,s\in\mathbb{R}.$ Then, we have for $i=1,2,$
	$$
	\|\Lambda ^{\beta /2}Q_i^{1/2}\|_{HS}^2=\|\Lambda ^{(\beta-s) /2}\|_{HS}^2=\sum_{j=1}^\infty\lambda_j^{\beta-s} 
	\approx \sum_{j=1}^\infty j^{\frac{2}{d}(\beta-s)},
	$$
	which is finite if and only if $\beta< s-\frac{d}{2},$ where $d$ is the dimension of the spatial domain $\mathcal{O}.$
	In the example above in \eqref{eqn_num1}, $d=1.$ Hence, we require $\beta<s-\frac{1}{2}.$

\begin{figure}[htbp]
    \centering
    \begin{minipage}{0.45\textwidth}
        \centering
        \includegraphics[width=\linewidth]{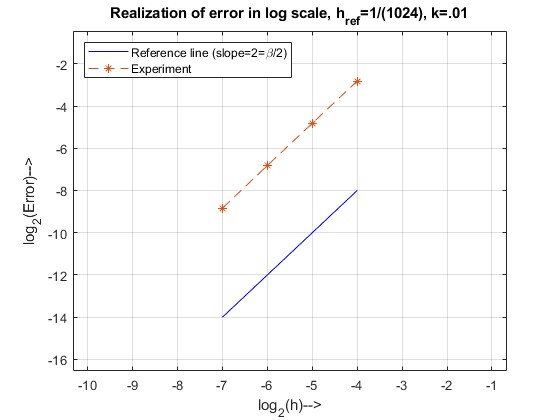} 
        \caption{The order of strong convergence in $L^2$-norm for deterministic problem}
        \label{fig1}
    \end{minipage}
    \hspace{0.05\textwidth} 
    \begin{minipage}{0.45\textwidth}
        \centering
        \includegraphics[width=\linewidth]{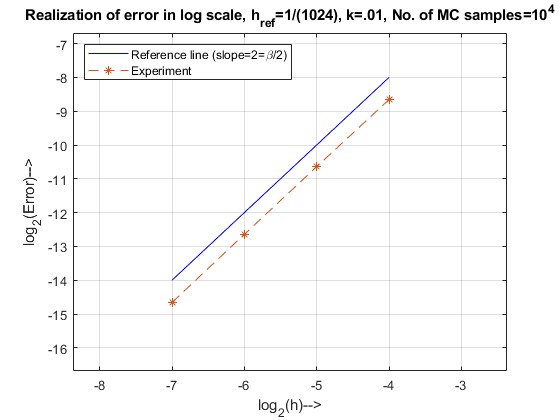} 
        \caption{The order of strong convergence in $L^2$-norm for stochastic problem}
        \label{fig2}
    \end{minipage}
\end{figure}
		In the numerical experiment, we have considered two cases:
		\begin{enumerate}[(i)]
			\item Deterministic linear Schr\"{o}dinger equation: $\beta=4, d=1,$ see Figure \ref{fig1}.
			\item Stochastic linear  Schr\"{o}dinger equation: $\beta=4, d=1,$ hence, $s>4+\frac{1}{2}.$ We choose $s=4+\frac{1}{2}+0.001$, see Figure \ref{fig2}.
		\end{enumerate}
		In above, we take $h_{\text{ref}}=2^{-10}$(step step for reference solution) and 
		$k=0.01$(time step) and $10^4$ Monte-Carlo samples for sampling in the stochastic case.

	\section{Conclusion}
	In this article, we have studied semi-discrete (in spatial variable) finite element approximations of stochastic linear Schr\"{o}dinger equation driven by additive Wiener noise. In a future work, we plan study stochastic semi-linear Schr\"{o}dinger equation driven by multiplicative Wiener noise and also strong convergence of fully (both in space and time) discretized model. Also, we will study weak convergence of the numerical approximation.
	
	\medskip
	\noindent
	{\bf{{Acknowledgement:}}} We express our gratitude to the Department of Mathematics and
	Statistics at the Indian Institute of Technology Kanpur for providing a conductive research
	environment. For this work,  M. Prasad is grateful
	for the support received through MHRD, Government of India (GATE fellowship). M. Biswas acknowledges support from IIT Kanpur. S. Bhar acknowledges the support from the SERB MATRICS grant (MTR/2021/000517), Government of India.

	
	\bibliographystyle{plain}

\begin{thebibliography}{10}

\bibitem{DCohen16}
Rikard Anton, David Cohen, Stig Larsson, and Xiaojie Wang.
\newblock Full discretization of semilinear stochastic wave equations driven by multiplicative noise.
\newblock {\em SIAM J. Numer. Anal.}, 54(2):1093--1119, 2016.

\bibitem{Brenner}
Susanne~C. Brenner and L.~Ridgway Scott.
\newblock {\em The mathematical theory of finite element methods}, volume~15 of {\em Texts in Applied Mathematics}.
\newblock Springer, New York, third edition, 2008.

\bibitem{courantL}
Thierry Cazenave.
\newblock {\em Semilinear {S}chr\"odinger equations}, volume~10 of {\em Courant Lecture Notes in Mathematics}.
\newblock New York University, Courant Institute of Mathematical Sciences, New York; American Mathematical Society, Providence, RI, 2003.

\bibitem{chuchuhong}
Chuchu Chen, Jialin Hong, and Lihai Ji.
\newblock Mean-square convergence of a symplectic local discontinuous {G}alerkin method applied to stochastic linear {S}chr\"odinger equation.
\newblock {\em IMA J. Numer. Anal.}, 37(2):1041--1065, 2017.

\bibitem{pchow}
Pao-Liu Chow.
\newblock {\em Stochastic partial differential equations}.
\newblock Chapman \& Hall/CRC Applied Mathematics and Nonlinear Science Series. Chapman \& Hall/CRC, Boca Raton, FL, 2007.

\bibitem{ciarlet}
Philippe~G. Ciarlet.
\newblock {\em The finite element method for elliptic problems}, volume~40 of {\em Classics in Applied Mathematics}.
\newblock Society for Industrial and Applied Mathematics (SIAM), Philadelphia, PA, 2002.
\newblock Reprint of the 1978 original [North-Holland, Amsterdam; MR0520174 (58 \#25001)].

\bibitem{DCohen13}
David Cohen, Stig Larsson, and Magdalena Sigg.
\newblock A trigonometric method for the linear stochastic wave equation.
\newblock {\em SIAM J. Numer. Anal.}, 51(1):204--222, 2013.

\bibitem{Dcohen12}
David Cohen and Xavier Raynaud.
\newblock Convergent numerical schemes for the compressible hyperelastic rod wave equation.
\newblock {\em Numer. Math.}, 122(1):1--59, 2012.

\bibitem{jianbo}
Jianbo Cui, Jialin Hong, and Zhihui Liu.
\newblock Strong convergence rate of finite difference approximations for stochastic cubic {S}chr\"odinger equations.
\newblock {\em J. Differential Equations}, 263(7):3687--3713, 2017.

\bibitem{prato}
Giuseppe Da~Prato and Jerzy Zabczyk.
\newblock {\em Stochastic equations in infinite dimensions}, volume~44 of {\em Encyclopedia of Mathematics and its Applications}.
\newblock Cambridge University Press, Cambridge, 1992.

\bibitem{dautray}
Robert Dautray and Jacques-Louis Lions.
\newblock {\em Mathematical analysis and numerical methods for science and technology. {V}ol. 5}.
\newblock Springer-Verlag, Berlin, 1992.
\newblock Evolution problems. I, With the collaboration of Michel Artola, Michel Cessenat and H\'el\`ene Lanchon, Translated from the French by Alan Craig.

\bibitem{Feng}
Xiaobing Feng and Shu Ma.
\newblock Stable numerical methods for a stochastic nonlinear {S}chr\"odinger equation with linear multiplicative noise.
\newblock {\em Discrete Contin. Dyn. Syst. Ser. S}, 15(4):687--711, 2022.

\bibitem{kovacs9bit}
Matthias Geissert, Mih\'aly Kov\'acs, and Stig Larsson.
\newblock Rate of weak convergence of the finite element method for the stochastic heat equation with additive noise.
\newblock {\em BIT}, 49(2):343--356, 2009.

\bibitem{griffiths}
David~J Griffiths and Darrell~F Schroeter.
\newblock {\em Introduction to quantum mechanics}.
\newblock Cambridge university press, 2018.

\bibitem{pgrisvard}
Pierre Grisvard.
\newblock {\em Elliptic problems in nonsmooth domains}, volume~69 of {\em Classics in Applied Mathematics}.
\newblock Society for Industrial and Applied Mathematics (SIAM), Philadelphia, PA, 2011.
\newblock Reprint of the 1985 original [MR0775683], With a foreword by Susanne C. Brenner.

\bibitem{Igyongy}
Istv\'an Gy\"ongy.
\newblock Approximations of stochastic partial differential equations.
\newblock In {\em Stochastic partial differential equations and applications ({T}rento, 2002)}, volume 227 of {\em Lecture Notes in Pure and Appl. Math.}, pages 287--307. Dekker, New York, 2002.

\bibitem{kovacs10NA}
Mih\'aly Kov\'acs, Stig Larsson, and Fredrik Lindgren.
\newblock Strong convergence of the finite element method with truncated noise for semilinear parabolic stochastic equations with additive noise.
\newblock {\em Numer. Algorithms}, 53(2-3):309--320, 2010.

\bibitem{Kovacs10Siam}
Mih\'aly Kov\'acs, Stig Larsson, and Fardin Saedpanah.
\newblock Finite element approximation of the linear stochastic wave equation with additive noise.
\newblock {\em SIAM J. Numer. Anal.}, 48(2):408--427, 2010.

\bibitem{Amartin}
Andreas Martin, Sergei~M. Prigarin, and Gerhard Winkler.
\newblock Exact and fast numerical algorithms for the stochastic wave equation.
\newblock {\em Int. J. Comput. Math.}, 80(12):1535--1541, 2003.

\bibitem{rockner}
Claudia Pr\'ev\^ot and Michael R\"ockner.
\newblock {\em A concise course on stochastic partial differential equations}, volume 1905 of {\em Lecture Notes in Mathematics}.
\newblock Springer, Berlin, 2007.

\bibitem{MR2224753}
Llu\'is Quer-Sardanyons and Marta Sanz-Sol\'e.
\newblock Space semi-discretisations for a stochastic wave equation.
\newblock {\em Potential Anal.}, 24(4):303--332, 2006.

\bibitem{CRoth}
Christian Roth.
\newblock Weak approximations of solutions of a first order hyperbolic stochastic partial differential equation.
\newblock {\em Monte Carlo Methods Appl.}, 13(2):117--133, 2007.

\bibitem{Cschwab}
Christoph Schwab and Radu~Alexandru Todor.
\newblock Karhunen-{L}o\`eve approximation of random fields by generalized fast multipole methods.
\newblock {\em J. Comput. Phys.}, 217(1):100--122, 2006.

\bibitem{Thomee}
Vidar Thom\'{e}e.
\newblock {\em Galerkin finite element methods for parabolic problems}, volume~25 of {\em Springer Series in Computational Mathematics}.
\newblock Springer-Verlag, Berlin, second edition, 2006.

\bibitem{JBWalsh}
John~B. Walsh.
\newblock An introduction to stochastic partial differential equations.
\newblock In {\em \'Ecole d'\'et\'e{} de probabilit\'es de {S}aint-{F}lour, {XIV}---1984}, volume 1180 of {\em Lecture Notes in Math.}, pages 265--439. Springer, Berlin, 1986.

\bibitem{JBWalsh6}
John~B. Walsh.
\newblock On numerical solutions of the stochastic wave equation.
\newblock {\em Illinois J. Math.}, 50(1-4):991--1018, 2006.

\bibitem{Whittaker1989}
Edmund Whittaker.
\newblock {\em A History of the Theories of Aether and Electricity: Vol. I: The Classical Theories; Vol. II: The Modern Theories, 1900-1926}, volume~1.
\newblock Courier Dover Publications, 1989.

\bibitem{YYanBit4}
Yubin Yan.
\newblock Semidiscrete {G}alerkin approximation for a linear stochastic parabolic partial differential equation driven by an additive noise.
\newblock {\em BIT}, 44(4):829--847, 2004.

\bibitem{YYanSiam5}
Yubin Yan.
\newblock Galerkin finite element methods for stochastic parabolic partial differential equations.
\newblock {\em SIAM J. Numer. Anal.}, 43(4):1363--1384, 2005.

\end{thebibliography}

\end{document}